\documentclass[a4paper,11pt]{amsart}
\usepackage{etex}
\pdfoutput=1
\usepackage[a4paper,margin=1in]{geometry}

\usepackage{amssymb,amsfonts,amsmath,amsthm}
\usepackage{latexsym,amsrefs,mathtools,stmaryrd}
\usepackage{mathrsfs}
\usepackage{xcolor}
\usepackage{tikz}
\usepackage{tikz-cd}
\usepackage{verbatim}
\usepackage{doi}
\usepackage{booktabs}
\usepackage{longtable}

\usepackage[section]{placeins}

\allowdisplaybreaks[1]

\def\be{\begin{equation}}
	\def\ee{\end{equation}}
\def\bes{\begin{equation*}}
	\def\ees{\end{equation*}}
\def\={\; = \;}
\def\+{\, + \,}
\def\m{\, - \,}

\def\:{\; \colon \;}
\def\eps{\varepsilon}
\def\sm{\smallsetminus}

\newcommand*\bmat[4]{\begin{bmatrix}#1&#2\\#3&#4\end{bmatrix}}

\newcommand*\smat[4]{\begin{smallmatrix}#1&#2\\#3&#4\end{smallmatrix}}

\newcommand{\CC}{\mathbb{C}}
\newcommand{\RR}{\mathbb{R}}
\newcommand{\QQ}{\mathbb{Q}}
\newcommand{\ZZ}{\mathbb{Z}}
\newcommand{\PP}{\mathbb{P}}
\newcommand{\NN}{\mathbb{N}}

\newcommand{\Hf}{\mathfrak{H}}

\newcommand{\An}{\mathcal{V}}
\newcommand{\MM}{\mathcal{M}}
\newcommand{\RM}{\mathcal{R}}
\newcommand{\FR}{\mathcal{F}}
\newcommand{\JR}{\mathcal{J}}
\newcommand{\Pp}{\mathcal{P}}
\newcommand{\Aa}{\mathcal{A}}
\newcommand{\Bb}{\mathcal{B}}
\newcommand{\Oo}{\mathcal{O}}

\newcommand{\Ct}{\mathscr{C}}
\renewcommand{\phi}{\varphi}
\renewcommand{\t}{\tau}
\newcommand{\s}{\sigma}
\newcommand{\vr}{\varrho}

\newcommand{\Li}{{\rm Li}}

\newcommand{\im}{\operatorname{Im}}
\newcommand{\re}{\operatorname{Re}}
\newcommand{\sgn}{\operatorname{sgn}}
\newcommand{\PSL}{\operatorname{PSL}}
\newcommand{\SL}{\operatorname{SL}}

\newcommand{\Mod}[1]{\ (\text{mod}\ #1)}
\newcommand{\Red}{\mathrm{Red}}
\newcommand{\ol}[1]{\overline{#1}}

\newcommand{\ds}{\displaystyle}

\newcommand{\cycle}[1]{\left(\mkern-4mu\left(#1\right)\mkern-4mu\right)}
\newcommand{\cyclesq}[1]{\left[\mkern-2mu\left[#1\right]\mkern-2mu\right]}

\newtheorem{theorem}{Theorem}
\newtheorem{proposition}{Proposition}

\newtheorem*{corollary*}{Corollary}

\title{Arithmetic properties of the Herglotz function}

\date{}

\author[Danylo Radchenko]{Danylo Radchenko}
\address{ETH Zurich, Mathematics Department, R\"amistrasse 101, 8092 Z\"urich, 
Switzerland}
\email{danradchenko@gmail.com}

\author[Don~Zagier]{Don~Zagier}
\address{\parbox{\linewidth}{Max Planck Institute for Mathematics, Vivatsgasse 7, 
53111 Bonn, Germany\\ International Centre for Theoretical Physics, Strada Costiera 
11, 34151 Trieste, Italy}}
\email{dbz@mpim-bonn.mpg.de}

\begin{document}
\maketitle

\begin{abstract}
	In this paper we study two functions $F(x)$ and $J(x)$, originally found by 
	Herglotz in 1923 and later rediscovered and used by one of the authors in 
	connection with the Kronecker limit formula for real quadratic fields. 
    We discuss many interesting properties of these functions, including special 
    values at rational or quadratic irrational arguments as rational linear 
    combinations of dilogarithms and products of logarithms, functional equations 
    coming from Hecke operators, and connections with Stark's conjecture. We also 
    discuss connections with 1-cocycles for the modular group $\PSL(2,\ZZ)$.
\end{abstract}

\setcounter{tocdepth}{1}
\tableofcontents

\section{Introduction}
Consider the function
	\be \label{eq:jdef} 
	J(x)\=\int_0^1\frac{\log(1+t^x)}{1+t}\,dt\,,
	\ee
defined for $x>0$. Some years ago, Henri~Cohen \footnote{who learned it from 
H.~Muzzafar (Montreal), see~\cite{C}*{Ex.60, p.~902--903}} showed one of 
the authors the identity
	\[J(1+\sqrt{2})\=
	-\frac{\pi^2}{24}\+\frac{\log^2(2)}{2}
	\+\frac{\log(2)\log(1+\sqrt{2})}{2}\,.\]
In this note we will give many more similar identities, like 
	\[J(4+\sqrt{17})\=
    -\frac{\pi^2}{6}\+\frac{\log^2(2)}{2}\+\log(2)\log(4+\sqrt{17})\]
and
	\[J\bigg(\frac25\bigg)\=
	\frac{11\pi^2}{240}\+\frac{3\log^2(2)}{4}
	\m2\log^2\bigg(\frac{\sqrt{5}+1}{2}\bigg)\,.\]
We will also investigate the connection to several other topics, such as the 
Kronecker limit formula for real quadratic fields, Hecke operators, Stark's 
conjecture, and cohomology of the modular group $\PSL_2(\ZZ)$.

The function $J(x)$ can be expressed via the formula
	\be \label{eq:jfrel}
	J(x)\=F(2x)\m2F(x)\+F(x/2)\+\frac{\pi^2}{12x}
	\ee
in terms of the function 
	\be\label{eq:fdef}
	F(x) \= \sum_{n=1}^{\infty}\frac{\psi(nx)-\log(nx)}{n}
	\qquad \bigl(x\in\CC'\coloneqq \CC\sm (-\infty,0]\,\bigr)\,,
	\ee
which is therefore in some sense more fundamental. 
Here $\psi(x) = \frac{\Gamma'(x)}{\Gamma(x)}$ is the digamma function.
The function~$F$ was introduced in~\cite{Z1}, where it was shown, in
particular, that it satisfies the two functional equations
	\begin{align}
	\label{eq:3term} 
	\begin{split}
		F(x)\m F(x+1)\m F\Big(\frac{x}{x+1}\Big)
		&\= -F(1)\+\Li_2\Big(\frac{1}{1+x}\Big)\,,\\
		F(x)\+F\Big(\frac{1}{x}\Big)&\=2F(1)\+\frac{\log^2 x}{2}
		\m\frac{\pi^2(x-1)^2}{6x}\,,
	\end{split}
	\end{align}
for all $x\in\CC'=\CC\sm(-\infty,0]$, where $\Li_2(x)$ is Euler's dilogarithm 
function. A function essentially equivalent to $F(x)$ was also studied by Herglotz 
in~\cite{H} and for this reason we will call $F$ the Herglotz function.
In fact, Herglotz also introduced the function $J(x)$ and found explicit 
evaluations~\cite[Eq.(70a-c)]{H} of $J(x)$ for $x=4+\sqrt{15}$, $6+\sqrt{35}$, 
and $12+\sqrt{143}$. Several other identities of this kind were found by 
Muzzafar and Williams~\cite{MW}, together with some sufficient conditions on~$n$
under which one can evaluate $J(n+\sqrt{n^2-1})$. We give a more systematic treatment 
of identities of this type in Section~\ref{sec:quadratic}.

\section{Elementary properties}
\subsection*{\indent Integral representations}
Using the well-known integral formula for $\psi(x)$ (see~\cite[p.~247]{WW})
	\[\psi(x)\=\int_{0}^{\infty}
	\Big(\frac{e^{-t}}{t}-\frac{e^{-xt}}{1-e^{-t}}\Big)\,dt\,,\]
we see that the function $F$ also has the integral representations
	\be \label{eq:integral}
		F(x) \= \int_0^\infty  
		\Big(\frac{1}{1-e^{-t}}-\frac{1}{t}\Big)
		\log(1-e^{-tx})\,dt 
		\= \int_0^1 
		\Big(\frac{1}{1-y}+\frac{1}{\log y}\Big)
		\log(1-y^x)\,\frac{dy}{y}\,.
	\ee
By differentiating the last integral and using the obvious identity
$\frac{d}{dx}\Li_m(y^x) = \Li_{m-1}(y^x)\log y$
we get also the following expression for the $n$-th derivative of~$F$
	\be
	\label{eq:integral_derivative}
	F^{(n)}(x)\=-\int_{0}^{1} \Big(\frac{1}{1-y} +\frac{1}{\log y}\Big)
	\log^n(y)\,\Li_{1-n}(y^x)\,\frac{dy}{y}\,.
	\ee
We will give another useful integral representation of~$F$ in 
Section~\ref{sec:period}.

\subsection*{\indent Asymptotic expansions.}
The Herglotz function has the asymptotic expansion
	\[F(x)\=F(1)\m\Li_2(1-x)
	\+\sum_{n=1}^{\infty}\frac{B_n\,\zeta(n+1)}{n}(x-1)^n(1-x^{-n})\]
as $x\to1$, the asymptotic expansion
	\[F(x) \;\sim\; -\frac{\pi^2}{12x} 
	\m\sum_{n=2}^{\infty}\frac{B_{n}\,\zeta(n+1)}{n}\,x^{-n}\,,\]
as $x\to\infty$, and the asymptotic expansion
	\[F(x) \;\sim\; -\frac{\zeta(2)}{x}\+\frac{\log^2x}{2}\+2\zeta(2)+2F(1)
	\+\sum_{n=1}^{\infty}\frac{B_{n}\,\zeta(n+1)}{n}\,x^{n}\,,\]
as $x\to 0$. In terms of the formal power series
	\[\mathfrak{Z}(x) \= 
	\sum_{n=1}^{\infty}\zeta(1-n)\,\zeta(1+n)\,x^n\;\in\RR[[x]]\,,\]
these can be rewritten as
	\be \label{eq:taylor}
	F(x) \;\sim\; 
    \begin{cases}
		\ds -\mathfrak{Z}(x)\+\frac{\log^2x}{2}\+2F(1)\m\zeta(2)(x-2+x^{-1})\,,
		\quad &x\to 0\,,\\[5pt]
		\ds \mathfrak{Z}\Big(\frac{1-x}{x}\Big)\m\mathfrak{Z}(1-x)\m\Li_2(1-x)+F(1)\,,
		\quad &x\to 1\,,\\[5pt]
		\ds \mathfrak{Z}\Big(\frac{1}{x}\Big)\,,
		\quad &x\to \infty\,.
	\end{cases}
	\ee
The expansion at $x\to\infty$ follows from the asymptotic expansion for the digamma 
function $\psi(z)\sim \log z+\sum_{n=1}^{\infty}\zeta(1-n)z^{-n}$, while the two 
other expansions can be derived easily from the expansion at $\infty$ using the 
functional equations~\eqref{eq:3term}.

\subsection*{\indent Analytic continuation.}
The function $F$ is analytic in the cut complex plane~$\CC'$, because the 
series~\eqref{eq:fdef} converges locally uniformly there. From the formula
$\psi'(z) = \sum_{m=0}^{\infty}(z+m)^{-2}$ we get a convergent series 
representation for $F'$ in $\CC'$
	\be \label{eq:fdoublesum}
	F'(x)\= \sum_{n=1}^{\infty}
    \bigg(\sum_{m=0}^{\infty}\frac{1}{(nx+m)^2}-\frac{1}{nx}\bigg)\,.
	\ee
	
Let us look at this from a different point of view. The change of 
variables~$y\mapsto x=(\frac{1-y}{1+y})^2$ maps the unit disk $\{y\colon|y|<1\}$ 
bijectively and conformally onto $\CC'$. The analyticity of $F$ in $\CC'$ together 
with the second formula in~\eqref{eq:taylor} then implies that the power series
	\be \label{eq:gformula}
	G(y)\coloneqq F\Big(\Big(\frac{1-y}{1+y}\Big)^2\Big)-F(1)
	\+\Li_2\Big(\frac{4y}{(1+y)^2}\Big)
	\= \mathfrak{Z}\Big(\frac{4y}{(1-y)^2}\Big)
	\m\mathfrak{Z}\Big(\frac{4y}{(1+y)^2}\Big)
	\ee
has radius of convergence $\ge 1$. The power series $\mathfrak{Z}(y)$ is factorially 
divergent, but we can show directly that the difference of two $\mathfrak{Z}$'s 
appearing on the right-hand side of~\eqref{eq:gformula} is convergent in the unit 
disk by the following argument. First, note that $\mathfrak{Z}(x)+\zeta(2)x$ 
is equal to $-\sum_{n\ge 1}\frac{\gamma_0(x/n)}{n}$. 
Here $\gamma_0(x)=\sum_{n\ge 1}\frac{B_m}{m}x^m$ is the formal power series 
discussed in detail in~\cite[(A.13)]{Z5}, where it was shown that it satisfies
	\[\gamma_0\Big(\frac{x}{1-x}\Big)-\gamma_0(x) \= \log(1-x)+x\,.\]
From this we obtain
	\be \label{eq:gformula2}
	G(y)+\zeta(2)\frac{16y^2}{(1-y^2)^2}
	\,=\, \sum_{n=1}^{\infty}\frac{1}{n}\sum_{j=0}^{n-1}
	\phi\Big(\frac{\frac{4}{n}y}{1-2(\frac{2j}{n}-1)y+y^2}\Big)\,,
	\;\quad \phi(x)\coloneqq-\log(1-x)-x\,.
	\ee
Using $\phi(x)=\int_{0}^{1}(\frac{x}{1-tx}-x)dt$ and the generating series 
$\sum_{m\ge 0}U_m(t)\,x^m = \frac{1}{1-2tx+x^2}$ for Chebyshev polynomials 
of the second kind, we obtain
	\bes
	[y^{m+1}]\biggl( G(y)+\zeta(2)\frac{16y^2}{(1-y^2)^2}\biggl)
	\= \sum_{n=1}^{\infty}\frac{2}{n}
	\biggl[\int_{-1}^{1}U_m(t)\,dt
	-\frac{2}{n}\sum_{j=0}^{n-1}U_m\biggl(\frac{2j}{n}-1\biggr)\biggr]\,.
	\ees
Now breaking up the integral into~$n$ integrals over intervals of length~$2/n$ 
and using the estimate $|U_m'(t)|\le4m(m+1)$, $t\in[-1,1]$, we find that the 
right hand side is $O(m^2)$, and hence that the series $G(y)$ indeed has 
radius of convergence at least~$1$.
	
\subsection*{\indent Relation to the Dedekind eta function.}
The analytic extension of $F$ to $\CC'$ satisfies another important functional 
equation. To state it, and for later purposes, it is convenient to introduce the 
slightly modified function
	\be
	\label{eq:fbetter}
	\FR(x) \= F(x)-F(1)
	\+\frac{\pi^2}{12}\big(x-2+x^{-1}\big)\m\frac{\log^2 x}{4}\,.
	\ee
\begin{proposition}
	For any~$z\in\Hf=\{z\colon \im(z)>0\}$ we have
	\[\FR(-z)\m\FR(z)\= 2\pi i\log((z/i)^{1/4}\eta (z))\,,\]
	where $\eta(z) = e^{\pi iz/12}\prod_{n\ge1}(1-e^{2\pi inz})$ is
	the Dedekind eta function, and the branch of
	$\log((z/i)^{1/4}\eta(z))$ is chosen to be real-valued on the 
	imaginary axis.
\end{proposition}
\begin{proof}
	This follows from $\psi(-x)-\psi(x)=x^{-1}+\pi\cot(\pi x)$. 
    Compare~\cite[p.~27]{VZ}.
\end{proof}
As a corollary, $\FR(x)$ cannot be extended analytically across the negative real 
axis either from above or from below.
	
\subsection*{\indent Relation to the Rogers dilogarithm.}
The functional equation~\eqref{eq:3term} already shows that $F$ is intimately related 
to the dilogarithm function~$\Li_2$. If we use the modified function $\FR$ 
from~\eqref{eq:fbetter} instead, we get nicer identities involving the Rogers 
dilogarithm. Recall that the Rogers dilogarithm is defined by
	\be \label{eq:rogersdef1}
	L(x) \= \Li_2(x) + \frac{1}{2}\log(x)\log(1-x)\,,
	\quad x\in[0,1]\,,
	\ee
and extended to $\RR$ by
	\be \label{eq:rogersdef2}
	L(x) \= \begin{cases}
	\tfrac{\pi^2}{3} - L(\tfrac{1}{x})\,, 
	\quad x>1\,, \\
	-L(\tfrac{x}{x-1})\,,\quad\;\, x<0\,.\end{cases} 
	\ee
The key property of $L$ is that it defines a continuous map from 
$\mathbb{P}^1(\RR)$ to $\RR/\frac{\pi^2}{2}\ZZ$ that satisfies
	\be \label{eq:5term1}
	L(x)+L(y)-L(xy)-L\bigg(\frac{x(1-y)}{1-xy}\bigg)
	-L\bigg(\frac{y(1-x)}{1-xy}\bigg) \= 0 \Mod{\pi^2/2}\,.
	\ee
Now, if we use $\FR$ instead of $F$, then we get ``clean'' versions of the 
functional equation~\eqref{eq:3term} with the Rogers dilogarithm in place of~$\Li_2$: 
for $x>0$ we have
	\begin{align}
	\label{eq:3term2} 
	\begin{split}
		\FR(x) \m\FR(x+1) \m\FR\Big(\frac{x}{x+1}\Big)
		&\= L(1/2)-L\Big(\frac{x}{x+1}\Big)\,,\\
		\FR(x) \+\FR\Big(\frac{1}{x}\Big)&\= 0\,.
	\end{split}
	\end{align}

\section{Functional equations related to Hecke operators}
\label{sec:funceq}
As it turns out, as well as the two functional equations~\eqref{eq:3term} 
the Herglotz function satisfies infinitely many further $1$-variable functional 
equations, one for each Hecke operator $T_n$. Before we make this statement precise, 
we look at a related elementary problem about relations satisfied by cotangent 
products.
	
\subsection{Identities between cotangent products.}
Let
	\[ C(x) \= \begin{cases}
		\cot \pi x\,,     \quad\; x\in\CC\sm\ZZ\,,\\
		\quad 0\,, \quad\quad\quad x\in\ZZ\,,
	\end{cases}  \]
so that~$C$ is periodic, vanishes on $\ZZ$, and is holomorphic away from $\ZZ$. 
Define 
	\[\Ct(x,y) \= C(x)C(y)+1\,,\] 
which will be our main interest in this subsection. Note that it is even, 
symmetric, and satisfies
	\be \label{eq:cotaneq}
	\Ct(x,y)\m\Ct(x,x+y)\m\Ct(x+y,y) \= 0
	\ee
whenever $xy(x+y)\ne 0$. (See~\cite{Z3} for an interpretation of this type of
functional equation in terms of the cohomology of~$\SL(2,\ZZ)$.) We are 
interested in more general relations of the form
	\[\sum_{i}\lambda_i \Ct(a_ix+b_iy,c_ix+d_iy) \= \mathrm{const}\,,\]
where $a_i,b_i,c_i,d_i$ are integers with $a_id_i-b_ic_i\ne0$.
	
Let~$\Gamma$ be the group $\PSL_2(\ZZ)$, which is generated by the 
matrices~$S=(\smat{0}{-1}{1}{0})$ and~$U=(\smat{1}{-1}{1}{0})$, with $S^2=U^3=1$,
and let~$T=US=(\smat{1}{1}{0}{1})$. For $n\in \NN$, let~$\MM_n$ be the set 
of~$2\times 2$ integer matrices of determinant~$n$, modulo $\{\pm1\}$, and 
let $\RM_n=\QQ[\MM_n]$. We denote both the element of~$\MM_n$ corresponding to a 
matrix $(\smat abcd)$ of determinant~$n$ and the corresponding basis element 
of~$\RM_n$ by $[\smat abcd]$. Define $\MM=\bigcup_{n\in\NN}\MM_n$ and set $\RM = 
\QQ[\MM] =\bigoplus_{n\in\NN}\RM_n$, which is a graded ring.
This ring acts on the right on the vector space of even meromorphic functions 
on $\CC^2$ by the formula
	\[\Bigl( f \circ \sum_{i}\lambda_i [\smat{a_i}{b_i}{c_i}{d_i}]\Bigr) 
      \bigl(x,y\bigr) \= \sum_{i} \lambda_i\,f(a_ix+b_iy,c_ix+d_iy)\,.\]
The same formula defines the action of $\RM$ on the space~$\An$
of even complex-valued functions on~$\CC^2/\ZZ^2$. We also define 
$\delta_1,\delta_2\in\An$ by
	\[\delta_1(x,y) \= C(x)\delta(y)\,,\qquad
	  \delta_2(x,y) \= \delta(x)\delta(y)\,, \]
where $\delta=\delta_{\ZZ}\colon\CC\to\{0,1\}$ is the characteristic function of 
$\ZZ$.
	
The functional equations that we are interested in are elements of the set
	\[\mathcal{I} \= \{\xi\in\RM \mid \Ct\circ\xi = \mathrm{const}\} \,,\]
which is a right ideal in $\RM$. First, we give an algebraic description
of~$\mathcal{I}$.
\begin{proposition}\label{prop:algcrit}
	An element $\xi$ of $\RM$ lies in the ideal $\mathcal{I}$ 
	if and only if 
	\[(1-S)\xi \,\in\, \ker(\Phi_1)\cap\ker(\Phi_2)\,,\]
	where $\Phi_i\colon\RM\to\QQ(u,v)\otimes\An$, $i=1,2$,
	are homomorphisms defined for $\gamma=[\smat abcd]\in\MM$ by
	\[\Phi_1(\gamma) \= \frac{1}{cu+dv}\otimes \delta_1\circ\gamma\,.
	\quad\quad 
    \Phi_2(\gamma) \= \frac{c\det(\gamma)^{-1}}{cu+dv}\otimes \delta_2\circ\gamma\,,\]
\end{proposition}
\begin{proof}
	From the definitions of $C$ and $\delta$ it follows that
	\[C(x+\eps) \= \frac{\delta(x)}{\pi \eps}+C(x)+O(\eps)\,, 
	\qquad x\in\CC/\ZZ,\;\;\eps\to 0\,, \]
	where we recall that~$C|_{\ZZ}=0$. 
	If we fix $\xi_0=(x_0,y_0)$ and set~$(x,y)=(x_0+\eps u,y_0+\eps v)$, then 
%	\begin{align*}
%		\Ct(ax+by,cx+dy) \= \eps^{-2}&\frac{\pi^{-2}}{(au+bv)(cu+dv)}
%		\,(\delta_1|\gamma)(x_0,y_0)\\
%		\+\eps^{-1}&\Big(\frac{\delta_2(ax_0+by_0,cx_0+dy_0)}{cu+dv}
%		-\frac{\delta_2(-cx_0-dy_0,ax_0+by_0)}{au+bv}\Big)+O(1) \,.
%	\end{align*}
	\[\Ct(ax+by,cx+dy) \= 
	 \frac{1}{\pi^2v\eps^2}\,\Phi_2\bigl((1-S)\gamma\bigr)(x_0,y_0)
	+\frac{1}{\pi\eps}\,\Phi_1\bigl((1-S)\gamma\bigr)(x_0,y_0) + O(1)\,.\]
	Thus the function~$f\coloneqq\Ct\circ\xi$ is continuous 
	if and only if $(1-S)\xi \in \ker(\Phi_1)\cap\ker(\Phi_2)$.
	When this happens, the function $f$ for fixed $y$ is holomorphic,
	$1$-periodic, and bounded when $|\im(x)|\to\infty$, and hence
	constant by Liouville's theorem. Applying the same argument to~$f$ as a 
	function of~$y$, we get that~$f(x,y)$ is constant.
%	
%	Thus the function~$f\coloneqq\Ct\circ\xi$ is continuous 
%	(and hence holomorphic) if and only if 
%	$(1-S)\xi \in \ker(\Phi_1)\cap\ker(\Phi_2)$.
%	Since for a fixed value of $y$ the function $x\mapsto f(x,y)$ 
%	is bounded when $\mathrm{Im}(x)\to\pm\infty$ and $1$-periodic, 
%	we get that $f(x,y)$ is independent of~$x$ by Liouville's theorem.
%	Applying the same argument to~$f$ as a function of~$y$, we get 
%	that~$f(x,y)$ is constant.
\end{proof}
	
We now use the criterion in Propostion~\ref{prop:algcrit} to write down nontrivial 
elements in the ideal of relations $\mathcal{I}$ using Hecke operators. 
Following~\cite{CZ}, we say that an element $\widetilde T_n\in \RM_n$
($n\in\NN$) ``acts like the $n$-th Hecke operator on periods'' if it satisfies
	\be
	\label{eq:heckelike}(1-S)\widetilde T_n \= T_n^{\infty}(1-S)\+(1-T)Y\,,
	\ee
for some $Y\in\RM_n$, where $T_n^{\infty}$ is the usual representative of 
the $n$-th Hecke operator
	\[T_n^{\infty} \= 
	\sum_{ad=n}\sum_{0\le b<d}\bmat ab0d \;\in\; \QQ[\MM_n]\,.\]
It was shown in~\cite{CZ} that such elements exist for every $n\in\NN$
and satisfy $r_f|_{2-k}\widetilde{T}_n = r_{f|_{k}T_n}$ for the period 
polynomial $r_f$ of any holomorphic modular form of weight~$k$, whence the name.

\begin{theorem} \label{thm:cotangentfe}
	Suppose that $\widetilde T_n\in \RM_n$ acts like the $n$-th Hecke operator on 
    periods. Then
	\be \label{eq:cotanfeq}
	(\Ct \circ \widetilde T_n)(x,y) 
    \m \sum_{l|n}l\,\Ct(lx,ly) \= c(\widetilde{T}_n)\,, 
    \ee
	where $c\colon\RM\to\ZZ$ is the group homomorphism defined on generators by
	\bes
	\label{eq:constvalue}
	\bmat{a}{b}{c}{d}
	\;\mapsto\;
	\begin{cases}
		1-\sgn(a+b)\,\sgn(c+d)\,, & a+b\ne 0\,,\;\;c+d\ne 0\,,\\
		1-\sgn(a+b)\,\sgn(d)\,,   & a+b\ne 0\,,\;\;c+d  = 0\,,\\
		1-\sgn(b)\,\sgn(c+d)\,,   & a+b  = 0\,,\;\;c+d\ne 0\,.
	\end{cases}
	\ees
\end{theorem}
\begin{proof}
	We apply the criterion from Proposition~\ref{prop:algcrit}. 
	Since $\delta_1\circ(1-T)=\delta_2\circ(1-T)=0$, the subspace
	$(1-T)\RM_n$ lies in both $\ker\Phi_1$ and $\ker\Phi_2$,
	and therefore $\Phi_i((1-S)\widetilde{T}_n) = \Phi_i(T_n^{\infty}(1-S))$. 
    Next, we calculate
		\begin{align*}
		n^{-1}\sum_{ad=n}\sum_{0\le b<d}\delta(ax+by)\delta(dy) \=
		\sum_{ad'd''=n}(ad')^{-1}\mathrm{den}(y,d')\delta(ad'x) \\
		\=\sum_{d''|n}d''n^{-1}\delta(xn/d'')\delta(yn/d'') \=
		\sum_{l|n}l^{-1}\delta(lx)\delta(ly)
		\end{align*}
	and 
		\begin{align*}
		\sum_{ad=n}\sum_{0\le b<d}
		d^{-1}C(ax+by)\delta(dy) \=
		\sum_{ad'd''=n}\mathrm{den}(y,d')C(ad'x) \\
		\=\sum_{d''|n}C(xn/d'')\delta(yn/d'') \=
		\sum_{l|n}C(lx)\delta(ly)\,,
		\end{align*}
	where $\mathrm{den}(x,d)$ is the function that equals $1$ if $x$ is a rational 
	number with denominator $d$, and~$0$ otherwise. (We have used the 
	identity $\sum_{j=0}^{m-1}C(x+j/m)=mC(mx)$.) Using these we compute
		\begin{align*}
		\Phi_1\big(T_n^{\infty}(1-S)\big) \= 
		\sum_{ad=n}\sum_{0\le b<d}
		\Big(\frac{d^{-1}C(ax+by)\delta(dy)}{v}
		-\frac{d^{-1}C(bx-ay)\delta(dx)}{u}\Big)\\
		\= \sum_{l|n}
		l\,\Big(\frac{C(lx)\delta(ly)}{lv}-\frac{C(-ly)\delta(lx)}{lu}\Big)
		\= \Phi_1\big(\sum_{l|n}l(\smat l00l)(1-S)\big)\,.
		\end{align*}
	and
		\begin{align*}
		\Phi_2\big(T_n^{\infty}(1-S)\big) \= 
		n^{-1}\sum_{ad=n}\sum_{0\le b<d}
		\Big(-\frac{\delta(bx-ay)\delta(dx)}{u}\Big) 
		\= \sum_{l|n}l^{-1} \Big(-\frac{\delta(-ly)\delta(lx)}{u}\Big)\\
		\= \Phi_2\big(\sum_{l|n}l(\smat l00l)(1-S)\big)
		\end{align*}
	This implies that the left-hand side of~\eqref{eq:cotanfeq} is constant.
	To get the stated value of this constant we set $x=it$, $y=i(1+\eps)t$, 
	and take the limit of the LHS in~\eqref{eq:cotanfeq} 
    as $t\to+\infty$, $\eps\to 0+$.
\end{proof}

\subsection{Functional equations for $\FR$.}
We will now use the cotangent functional equations of Theorem~\ref{thm:cotangentfe} 
to obtain functional equations for the Herglotz function. Recall the standard slash 
action of $\RM$ in even weight $k$, defined on the 
generators $\gamma=(\smat abcd) \in \MM$ by
	\[(f|_k\gamma)(x) \= 
	\frac{(ad-bc)^{k/2}}{(cx+d)^k}\,f\Bigl(\frac{ax+b}{cx+d}\Bigr)\,.\]
\begin{theorem} \label{thm:heckefeq}
    Suppose that $\widetilde T_n=\sum_{\gamma}\nu_{\gamma}\gamma$ acts 
    like the $n$-th Hecke operator on periods and that all matrices 
    in $\widetilde T_n$ have nonnegative entries. Then the modified Herglotz 
    function~$\FR$ satisfies
    	\be  \label{eq:heckefeq}
    	\bigl(\FR\bigl|_{0}\widetilde T_n\bigr)(x) \m \s(n)\FR(x)
    	\m \sum_{\gamma(\infty)\ne\infty} \nu_{\gamma}
    	L\bigl(\gamma(x)/\gamma(\infty)\bigr) 
    	\= A_n\log x+\mathrm{const}\qquad (x>0)\,,
    	\ee
    where $L(x)$ is the Rogers dilogarithm defined in equation~\eqref{eq:rogersdef1} 
    and $A_n = \frac{1}{2}\sum_{l|n}l\log(l^2/n)$.
\end{theorem}
\begin{proof}
We use the same method as in~\cite{Z1}, 
namely, to characterize $F'(x)$ by the equation
	\[ A(x,s) \= \frac{x^{-1}}{s-1} 
    \+ \Big(F'(x)-\frac{\log x}{x}-\frac{\zeta(2)}{x^{2}}\Big) \+ O(s-1)\,, \]
where $A(x,s)$ is the function defined for $\re x>0$, $\re s>1$ by
	\[A(x,s) \= \int_{0}^{\infty}\frac{t^s\,dt}{(e^{xt}-1)(e^t-1)}\,.\]
Now, for $x,y>0$ we calculate
	\begin{align*}
	-\frac{(2\pi)^{s+1}}{4}\int_{0}^{\infty}\Ct(ixt,iyt)\,t^{s}\,dt
	&\= y^{-s-1}A(x/y,s) 
	+\tfrac{1}{2}\,\Gamma(s+1)\,\zeta(s+1)\,(x^{-s-1}+y^{-s-1}) \\
	&\= \frac{(xy)^{-1}}{s-1}
	+ y^{-2}F'\Big(\frac{x}{y}\Big)
	+\frac{\pi^2}{12}\Big(\frac{1}{y^{2}}-\frac{1}{x^{2}}\Big)
	-\frac{\log x}{xy} + O(s-1)\\
	&\= \frac{(xy)^{-1}}{s-1}
	+y^{-2}\FR'\Big(\frac{x}{y}\Big)-\frac{\log(xy)}{2xy} + O(s-1) \,.
	\end{align*}
(Note that $\Ct(ixt,iyt)$ is exponentially small for $t$ large.)
Since all entries of matrices in $\widetilde T_n$ are nonnegative,
we have $c(\widetilde T_n)=0$ (where $c$ is the homomorphism defined in 
Theorem~~\ref{thm:cotangentfe}), and equation~\eqref{eq:cotanfeq} then implies that 
for $x>0$
	\[\sum_{\gamma=(\smat abcd)\in\MM_n} \nu_{\gamma}\,
	\bigg((\FR'|_{2}\gamma)(x) - n\frac{\log((ax+b)(cx+d))}{2(ax+b)(cx+d)}\bigg)
    \= \sum_{l|n}\frac{n}{l} \Big(\FR'(x) - \frac{\log(l^2x)}{2x}\Big)\,.\]
Next, we integrate the above identity. The primitive of $(\FR'|_{2}\gamma)(x)$
is $\FR(\gamma x)$, while for the logarithmic term we use (assuming that $ad-bc=n$)
	\[n\int\frac{\log((ax+b)(cx+d))}{2(ax+b)(cx+d)}\,dx \= 
	L\Big(\frac{x+b/a}{x+d/c}\Big)
	+\frac{1}{4}\log^2(n(x+b/a))-\frac{1}{4}\log^2(n(x+d/c))\]
for $c>0$ and its variant
	\[n\int\frac{\log((ax+b)d)}{2(ax+b)d}\,dx\=
	\frac{1}{4}\log^2(n(x+b/a))\]
for $c=0$ (both statements are easily checked by differentiation).
Thus, modulo constants, the left-hand side of~\eqref{eq:heckefeq} 
is equal to
	\[\frac{1}{4}\bigg(\sum_{\gamma}\nu_{\gamma}\log^2\big(n(x-\gamma^{-1}(0))\big)
	-\sum_{\gamma^{-1}(\infty)\ne \infty}\nu_{\gamma}
        \log^2\big(n(x-\gamma^{-1}(\infty))\big)
	-\sum_{l|n}\frac{n}{l}\log^2(l^2x)\bigg) \,.\]
A simple calculation using~\eqref{eq:heckelike} shows that
the homomorphism $\phi\colon\RM\to\ZZ[\PP^1(\QQ)]$ defined on the 
generators by $\gamma\mapsto [\gamma^{-1}(0)]-[\gamma^{-1}(\infty)]$
satisfies $\phi(\widetilde{T}_n)=\s(n)([0]-[\infty])$. 
Thus, the above expression simplifies to
	\[\frac{1}{4}\Big(\s(n)\log^2(nx)-\sum_{l|n}\frac{n}{l}\log^2(l^2x)\Big)
	\= A_n\log x + \text{const}\,,\]
as claimed.
\end{proof}

The following proposition shows that a $\widetilde T_n$ consisting
of matrices with nonnegative entries exist for all~$n\in\NN$.
Similar elements were constructed by Manin~\cite{Ma} and others.
\begin{proposition}\label{prop:nonnegativehecke}
	The element $\widehat T_n\in\RM_n$ defined by
	\be\label{eq:tnhat}
	\widehat T_n \= \sum_{\substack{0\le c<a,\ 0\le b<d\\ ad-bc=n}}
	\bmat abcd \ee
	acts like the $n$-th Hecke operator on periods, i.e., there exists an element 
	$Y_n\in\RM_n$ such that
	\[(1-S)\widehat T_n \= T_n^{\infty}(1-S)+(1-T)Y_n\,.\]
\end{proposition}
\begin{proof}
Let us denote $\mathcal{P}_n = \{[\smat abcd]\in\MM_n\;|\; 0<c<a,\,0\le b<d\}$. 
It is easy to check that the map $\iota\colon\mathcal{P}_n\to\mathcal{P}_n$ given by
	\[\bmat abcd \;\mapsto\; 
	  \bmat {\lceil a/c\rceil d-b}{\lceil a/c\rceil c-a}{d}{c}\]
is a well-defined involution. Working modulo $(1-T)\RM_n$, 
which amounts to introducing the equivalence relation
$[\smat abcd] \equiv [\smat {a+kc}{b+kd}cd]$, $k\in\ZZ$, we calculate
    \[(1-S)\widehat T_n - T_n^{\infty}(1-S) \;\equiv\;
	\sum_{\substack{0<c<a\\ 0\le b<d}}\bmat abcd
    \m\sum_{\substack{0\le c<a\\ 0<b<d}}\bmat {-c}{-d}ab
	\= \sum_{\substack{0<c<a\\ 0\le b<d}}\left(\bmat abcd - \bmat{-b}{-a}dc\right),\]
where all the sums are over matrices of determinant~$n$, and to get the last equality 
we made the substitution $a\leftrightarrow d, b\leftrightarrow c$ 
in the second sum. Since $\iota([\smat abcd])\equiv [\smat {-b}{-a}dc]$,
the terms in the last sum cancel in pairs (modulo $(1-T)\RM_n$), 
and we conclude that $(1-S)\widehat T_n - T_n^{\infty}(1-S) \in (1-T)\RM_n$.
\end{proof}
An essentially equivalent version of~$\widehat T_n$ (with inequalities for rows 
instead of columns) was given by Merel in~\cite{Me}. Here are the first three 
nontrivial examples
\begin{align*}
	&\widehat T_2 \= \bmat 1002 + \bmat 1102 +\bmat 2001+\bmat 2011,\qquad
	\widehat T_3 \= \sum_{j=0}^{2}\bmat 1j03 +\sum_{j=0}^{2}\bmat 30j1 
	+\bmat 2112,\\
	&\widehat T_4 \= \sum_{j=0}^{3}\bmat 1j04 +\sum_{j=0}^{3}\bmat 40j1 
	+\bmat 2002+\bmat 2012+\bmat 2102+\bmat 2213+\bmat 3122.
\end{align*}

For the special case $\widetilde{T}_n=\widehat{T}_n$, Theorem~\ref{thm:heckefeq} 
gives the following explicit functional equation.
\begin{corollary*} \label{cor:heckefeqex}
	For all $n\in\NN$ and $x>0$ the modified Herglotz function~$\FR$ satisfies
   	\be \label{eq:heckefeqex}
   	\bigl(\FR\bigl|_{0}\widehat T_n\bigr)(x) \m \s(n)\FR(x)
   	\= \sum_{\substack{0<c<a\\ 0\le b<d\\ ad-bc=n}} 
   	\biggl[L\biggl(\frac{x+b/a}{x+d/c}\biggr)
          -L\biggl(\frac{1+b/a}{1+d/c}\biggr)\biggr] \+ A_n\log x\,, \ee
    where $L(x)$ is the Rogers dilogarithm, and 
    $A_n = \frac{1}{2}\sum_{l|n}l\log(l^2/n)$.
\end{corollary*}
\begin{proof}
	By Theorem~\ref{thm:heckefeq} equation~\eqref{eq:heckefeqex} holds
	up to an additive constant. Setting $x=1$ and using the functional equation
	$\FR(x)+\FR(1/x)=0$ together with the fact that
	$(\smat 0110)\widehat{T}_n(\smat 0110)=\widehat{T}_n$ 
	shows that~\eqref{eq:heckefeqex} holds for $x=1$, and hence the constant is~$0$.
\end{proof}

The functional equations corresponding to the three $\widehat{T}_n$'s given above are
\begin{subequations}
	\be 
	\tag{\theequation a}
	\FR(2x)+\FR(\tfrac{x}{2})+\FR(\tfrac{x+1}{2})+\FR(\tfrac{2x}{x+1})-3\FR(x)
	\= L(\tfrac{x}{x+1})-L(\tfrac{1}{2}) +\tfrac{1}{2}\log(2)\log(x)\,, 
	\label{eq:heckefe2}
	\ee
	\be
	\tag{\theequation b}
	\label{eq:heckefe3}
		\begin{split}
		\FR(3x)+\FR(\tfrac{3x}{x+1})+\FR(\tfrac{3x}{2x+1})
		+\FR(\tfrac{x}{3})+\FR(\tfrac{x+1}{3})+\FR(\tfrac{x+2}{3})
		+\FR(\tfrac{2x+1}{x+2})-4\FR(x) \\
		\= L(\tfrac{2x+1}{2x+4})+L(\tfrac{x}{x+1})+L(\tfrac{2x}{2x+1})-
		L(1)-L(\tfrac{2}{3})+\log(3)\log(x)\,,
		\end{split}
	\ee
	\be
	\tag{\theequation c}
	\label{eq:heckefe4}
		\begin{split}
		&\textstyle\sum_{j=0}^{3}\big(\FR(\tfrac{x+j}{4})+\FR(\tfrac{4x}{jx+1})\big)
		+\FR(\tfrac{2x}{x+2})+\FR(\tfrac{2x+1}{2})
		+\FR(\tfrac{2x+2}{x+3})+\FR(\tfrac{3x+1}{2x+2})-6\FR(x)
		\\& \= \textstyle\sum_{j=1}^{3}L(\tfrac{jx}{jx+1})
		+L(\tfrac{x}{x+2})+L(\tfrac{x+1}{x+3})+L(\tfrac{3x+1}{3x+3})
		-4L(1)+L(\tfrac{2}{3})+3\log(2)\log(x)\,.
		\end{split}
	\ee
\end{subequations}
We will make use of the first and the third of these in Section~\ref{sec:quadratic} 
to obtain explicit evaluations of the function $J(x)$ for certain quadratic units~$x$.
	
We remark that there is no loss of generality in passing from the apparently more 
general functional equation~\eqref{eq:heckefeq} to the special 
case~\eqref{eq:heckefeqex}, because one can show that the functional equation coming 
from any other choice of $\widetilde{T}_n$ would be the same up to a linear 
combination of specializations 
of the functional equations~\eqref{eq:3term2}.

\section{Special values at positive rationals}
\label{sec:rational}
The special values of the Herglotz function at positive rationals turn out to always 
be expressible in terms of dilogarithms. Here it is more convenient to use the 
original function~$F(x)$ instead of~$\FR(x)$.
\begin{theorem}
	\label{thm:Fspecial}
	For any positive rational $x$ the value $F(x)-F(1)$ is a rational 
	linear combination of~$\Li_2$ with arguments belonging to the 
	cyclotomic field~$\QQ(e^{2\pi i x}, e^{2\pi i/x})$.
\end{theorem}
\begin{proof}
From the second integral representation in~\eqref{eq:integral} we get
	\[ F\Big(\frac{P}{Q}\Big)-F(1) 
	\= F\Big(\frac{P}{Q}\Big)-F\Big(\frac{P}{P}\Big)
	\=\int_{0}^{1}\Big(\frac{Q}{1-t^Q}-\frac{P}{1-t^P}\Big)
	\log(1-t^P)\,\frac{dt}{t}\,.  \]
Therefore, using the distribution relations
	\[\log(1-t^P) \= \sum_{\beta^{P}=1}\log(1-\alpha t)
	\qquad\mbox{ and }\qquad
	\frac{Q}{1-t^Q} \= 
	\sum_{\alpha^{Q}=1}\frac{1}{1-\beta t}\,,\]
we see that for all positive integers $P$, $Q$ one has
	\[ F\Big(\frac{P}{Q}\Big) - F(1) \= \sum_{\alpha^P=1}
    \Biggl(\sum_{\beta^Q=1,\,\beta\ne 1}\,{f(\alpha,\beta)} 
	\m\sum_{\beta^P=1,\,\beta\ne 1}\,{f(\alpha,\beta)}\Biggr)\,, \]
where
	\[f(\alpha,\beta)\=
	\int_{0}^{1}\frac{\log(1-\alpha t)}{t(1-\beta t)}\,dt\,.\]
The statement of the theorem follows using the identity
	\[f(\alpha,\beta) \= \Li_2\Big(\frac{\beta}{\beta-1}\Big) \m
	\Li_2\Big(\frac{\alpha-\beta}{1-\beta}\Big)\,,\]
which can be proved easily by differentiating both sides in~$\alpha$.
\end{proof}

We give three concrete examples of the theorem.

\noindent{\bf 1.} Setting $(P,Q)=(1,n)$ in the above proof and using 
$\Li_2(\tfrac{x}{x-1})+\Li_2(x)=-\tfrac{1}{2}\log^2(1-x)$ gives
	\be \label{eq:f1n}
	F(n)-F(1) \= -\frac{(n-1)(n-2)\pi^2}{24n}
	\+\frac{\log^2 n}{2} \+\frac{1}{2}\sum_{j=1}^{n-1}
	\log^2\Big(2\sin\frac{\pi j}{n}\Big)\,. \ee
Notice that this statement is stronger than Theorem~\ref{thm:Fspecial}, 
since we only get products of logarithms rather than dilogarithms. 
Conjecturally, the only positive rationals $x$ for which $F(x)-F(1)$ can reduce 
to products of logarithms are $x=n$ or $x=1/n$, since for other values of $x$ the 
corresponding formal combination of arguments of $\Li_2$ does not lie in the Bloch 
group of~$\ol\QQ$\,.
	
\noindent{\bf 2.} Combining~\eqref{eq:f1n} with the 3-term relation~\eqref{eq:3term} 
gives $F(\tfrac{n}{n+1})-\Li_2(\tfrac{n}{n+1})$ as a bilinear combination 
of logarithms of elements of $\QQ(\zeta_{n},\zeta_{n+1})$, 
where $\zeta_{n}=e^{2\pi i/n}$.

\noindent{\bf 3.} As a further example, which we will generalize in 
Section~\ref{sec:k2cocycle}, we have
	\be\label{eq:f25}
	F\Big(\frac{2}{5}\Big)-F(1) \= 
	\frac{1}{2}\Li_2\Big(\frac{4}{5}\Big)-\frac{\pi^2}{5}
    \+\log^2\Big(2\sin\frac{\pi}{5}\Big)\+\log^2\Big(2\sin\frac{2\pi}{5}\Big)
    \+\log(2)\log\Big(\frac{2}{5}\Big)\,.
	\ee

\smallskip
An argument similar to the one used in the proof of Theorem~\ref{thm:Fspecial} allows 
one to compute also the derivative of~$F$ at rational points.
\begin{proposition}
	For any coprime $p,q>0$ the difference $\tfrac{p}{q}F'(\tfrac{p}{q})-(1+\log(p))$
	is a linear combination of $\Li_2$ at $p$-th or $q$-th 
	roots of unity with coefficients in $\QQ(\zeta_p,\zeta_q)$.
\end{proposition}
\begin{proof}
Replacing the integral reprsentation~\eqref{eq:integral}  
by~\eqref{eq:integral_derivative} in the previous proof we obtain
	\[ \frac{P}{Q}F'\Big(\frac{P}{Q}\Big)-F'(1) \= \sum_{\alpha^P=1}
    \Biggl(\sum_{\beta^Q=1,\,\beta\ne 1}\,{g(\alpha,\beta)} 
	\m \sum_{\beta^P=1,\,\beta\ne 1}\,{g(\alpha,\beta)}\Biggr)\,, \]
where
	\[g(\alpha,\beta) \= -\int_0^1\frac{\alpha\log(y)}{(1-\alpha y)(1-\beta y)}\,dy
	\= \begin{cases}
		\ds \alpha\frac{\Li_{2}(\alpha)-\Li_{2}(\beta)}{\alpha-\beta}\,,
		\quad &\alpha\ne\beta\,,\\[5pt]
		\ds \Li_1(\alpha)\,,
	    \quad &\alpha=\beta\,.
	\end{cases}\]
The result then follows by noting that $F'(1)=1$, and that 
$\sum_{\alpha^P=1,\alpha\ne 1}\Li_1(\alpha)=-\log(P)$.
\end{proof}
Similarly, one can show that the value 
$(p/q)^kF^{(k)}(p/q)-F^{(k)}(1)+(-1)^k(k-1)!\log(p)$ 
is a combination of $\Li_m$, $m=2,\dots,k+1$ at $p$-th or $q$-th roots of unity with 
coefficients in $\QQ(\zeta_p,\zeta_q)$.

\section{Kronecker limit formula for real quadratic fields} \label{sec:klf}
The function $F$ appeared in~\cite{Z1} in a formulation of a so-called Kronecker 
limit formula for real quadratic fields. (The original Kronecker limit formula was 
the corresponding statement for imaginary quadratic fields.) This is a formula 
expressing the value
	\[\vr(\Bb) \= \lim_{s\to1}\bigg(D^{s/2}\zeta(\Bb,s)
    \m\frac{\log\eps}{s-1}\bigg)\,,\]
where $\Bb$ is an element of the narrow class group~\footnote{more precisely, the 
group of narrow ideal classes of invertible fractional $\Oo_D$-ideals, where two 
ideals belong to the same narrow class if their quotient is a principal ideal 
$\lambda O_D$ with $N(\lambda)>0$. For more details see~\cite{Z2} or~\cite{HK}}
of the quadratic order $\Oo_D = \ZZ + \ZZ\frac{D+\sqrt{D}}{2}$ of discriminant $D>0$, 
$\zeta(\Bb,s)$ is the corresponding partial zeta function (the sum of 
$N(\mathfrak{a})^{-s}$ over all invertible $\Oo_D$-ideals $\mathfrak{a}$ in the class 
$\Bb$), and $\eps=\eps_D$ is the smallest unit $>1$ in $\Oo_D$ of norm~$1$. The 
numbers~$\vr(\Bb)$ are of interest because the value at $s=1$ of the Dirichlet 
series $L_K(s,\chi) = \sum_{\mathfrak{a}} \chi(\mathfrak{a})/N(\mathfrak{a})^s$ 
equals $D^{-1/2}\sum_{\Bb}\chi(\Bb)\vr(\Bb)$ for any character $\chi$ on the 
narrow ideal class group. The formula proved in~\cite{Z1} is that
	\be\label{eq:kronlim} \vr(\Bb) \= \sum_{w\in{\Red}(\Bb)}P(w,w')\,,\ee
where the function $P(x,y)$ for $x>y>0$ is defined as
	\[P(x,y) \= F(x) -F(y) +\Li_2\Big(\frac yx\Big) -\frac{\pi^2}{6}+
	\log\frac xy\,\Big(\gamma -\frac12\log(x-y) +\frac14\log\frac xy\Big)\,,\]
and $\Red(\Bb)$ is the set of larger roots $w=\frac{-b+\sqrt{D}}{2a}$ of all 
reduced primitive quadratic forms $Q(X,Y)=aX^2+bXY+cY^2$ ($a,c>0,\ a+b+c<0$) of
discriminant $D$ which belong to the class $\Bb$. Recall that narrow ideal classes 
correspond to $\PSL_2(\ZZ)$-orbits on the set of integral quadratic forms: if 
$\mathfrak{b}=\ZZ w_1 + \ZZ w_2\in\Bb$ with $\frac{w_1w_2'-w_2w_1'}{\sqrt{D}}>0$, 
then the quadratic form $Q(X,Y)=\frac{N(Xw_1+Yw_2)}{N(\mathfrak{b})}$ is in the 
corresponding orbit. The set $\Red(\Bb)$ is finite and every element $w\in\Red(\Bb)$ 
satisfies $w>1>w'>0$. Let us also denote $l(\Bb) = \#\Red(\Bb)$.
	
The set $\Red(\Bb)$ can also be understood in terms of continued fractions. To any 
element~$w\in K$ with $w>w'$ (here we assume that $K$ is embedded into $\RR$) one can 
associate a continued fraction
	\[w \= b_1-\frac{1}{b_2-\ds\frac{1}{\ddots}}\,,\]
which by the standard theory of continued fractions is eventually periodic.
The sequence $\{b_j\}_{j\ge1}$ is periodic (without a pre-period) 
if and only if the number $w$ is \textit{reduced} (i.e., $w>1>w'>0$),
and then has period $l=l(\Bb)$. We call the $l$-tuple 
$\cycle{b_1,\dots,b_{l}}$ the cycle associated to $w$.
If we fix a narrow ideal class $\Bb$ and take 
any $\mathfrak{b}\in\Bb^{-1}$ such that 
$\mathfrak{b}=\ZZ+\ZZ w$ with $w$ reduced, then the
equivalence class of $\cycle{b_1,\dots,b_{l}}$ 
modulo cyclic permutations depends only on $\Bb$. The 
set $\Red(\Bb)$ is then simply $\{w_1,\dots,w_{l}\}$, where
	\[w_j \= b_j-\frac{1}{b_{j+1}-\ds\frac{1}{\ddots}}\,,\]
and both $b_j$ and $w_j$ depend only on $j\,(\text{mod}\ l)$. 
Thus we can restate~\eqref{eq:kronlim} as
	\be \label{eq:kronlim2} \vr(\Bb)\= \sum_{j\Mod{l}}P(w_j,w_j')\,. \ee
Using the fact that $\sum_{j\Mod{l}}\bigl(w_j+1/w_j\bigr)=\sum_{j\Mod{l}}\bigl(w_j'+1/w_{j}'\bigr)$ 
and $\prod_{j\Mod{l}}w_j=\eps$, we can  
rewrite~\eqref{eq:kronlim2} as
	\be\label{eq:kronlim3} \vr(\Bb) \= 2\gamma\log\eps 
    \m l(\Bb)\frac{\pi^2}{6} \+ \sum_{j\Mod{l}}\Pp(w_j,w_j')\,, \ee
where the function $\Pp(x,y)$, $x>y>0$ is now defined by a simpler formula
	\be \label{eq:Pdef} \Pp(x,y) \= \FR(x)\m\FR(y)\+L\Big(\frac{y}{x}\Big)\,. \ee

One can use the Kronecker limit formula~\eqref{eq:kronlim} to prove 
various properties of~$\vr(\Bb)$. 
For example, the functional equation
\eqref{eq:3term} was used in~\cite{Z1} to prove Meyer's formula
	\[\vr(\Bb) - \vr(\Bb^*) \= -\frac{\pi^2}{6}\big(l(\Bb)-l(\Bb^*)\big)\,,\]
where $\Bb^*=\Theta\Bb$, and $\Theta$ is the narrow ideal class of principal ideals 
of negative norm. The key observation used in that proof is equivalent to the identity
	\be \label{eq:p3term} \Pp(x,y) \= \widetilde\Pp(x-1,1-y)
	\m\widetilde\Pp\Big(1-\frac{1}{x},\frac{1}{y}-1\Big)\,,\qquad x>1>y>0\,, \ee
where $\widetilde\Pp(x,y)=\FR(x)-\FR(y)+L(-y/x)$,
that allows one to rewrite $\vr(\Bb)-2\gamma\log\eps+\frac{\pi^2}{6}l(\Bb)$ 
in terms of the continued fraction associated to the wide ideal class 
containing~$\Bb$. For details we refer to~\cite{Z1}, but we do state here 
a version of the Kronecker limit formula for wide ideal classes, since we 
will use it in Section~\ref{sec:quadratic} below. A number $x\in K$ 
is called \textit{reduced in the wide sense} if $x>1$, $0>x'>-1$, and to 
any such number one associates a cycle $\cyclesq{a_1,\dots,a_m}$, which is 
simply the periodic part of the regular continued fraction of $x$, i.e., 
$x=a_1+1/(a_2+1/(a_3+\dots))$ with $a_{i}=a_{i\,(\text{mod}\ m)}$ for $i\in\ZZ$.
We define $\Red^{W}(\Aa)$ to be the set of reduced numbers $x$
(in the wide sense) such that $\ZZ x+\ZZ$ is a fractional ideal in $\Aa$. 
Then $\Red^{W}(\Aa)=\{x_1,\dots,x_m\}$, where
	\[x_i \= a_i +\frac{1}{a_{i+1} +\ds\frac{1}{\ddots}}\,.\]
\begin{proposition}[\cite{Z1}*{Corollary~2, p.~179}]
	Let $\Aa$ be a wide ideal class in a real quadratic field $K$,
	$\eps_0>1$ the fundamental unit of~$K$, $x_1,\dots,x_m$ the elements of $K$ 
	satisfying $x_i>1$, $0>x_i'>-1$ and such that $\{1,x_i\}$ is a basis of some 
	ideal in $\Aa$, and $\cyclesq{a_1,\dots,a_m}$ the corresponding cycle of 
	integers. 
	Then
	\be \label{eq:kronlimwide}
	\vr(\Aa) \= 2\sum_{i=1}^{m}\widetilde{\Pp}(x_i,-x_i')
	\+4\gamma \log \eps_0\m\frac{\pi^2}{6}\sum_{i=1}^{m}(a_i-1)
	\ee
	with $\widetilde{P}$ as above.
\end{proposition}

\section{Special values at quadratic units}
\label{sec:quadratic}
While there does not seem to be a general formula, akin to
that of Theorem~\ref{thm:Fspecial}, that would express in closed form the 
individual values $\FR(x)$ for $x$ in real quadratic fields, there are many 
rational linear combinations of these values that can be evaluated. 
One way to obtain such identities is to specialize functional equations 
satisfied by $\FR(x)$, for example the $3$-term relation~\eqref{eq:3term2}, or 
the more complicated functional equations from Section~\ref{sec:funceq}.
A completely different source of identities, as was explained 
in~\cite[\S9]{Z1}, stems from the fact that if $\chi$ is a genus character on 
the narrow class 
group of a real quadratic field of discriminant $D$, then the special value
$L_K(1,\chi) = D^{-1/2}\sum_{\Bb}\chi(\Bb)\vr(\Bb)$ is equal to a product
$L_{D_1}(1)L_{D_2}(1)$, where the splitting $D=D_1D_2$ corresponds to the
genus character $\chi$, and $L_D(s)$ is the Dirichlet series
	\[L_D(s)=\sum_{n\ge1}\Bigl(\frac{D}{n}\Bigr)n^{-s}\,.\]
Since $L_D(1)$ can be evaluated explicitly as 
	\[L_D(1) \= 
	\begin{cases}
		\frac{2h\log\eps}{\sqrt{|D|}}\,,\qquad \;\; D>0\,,\\ 
		\frac{2\pi h}{w\sqrt{|D|}}\,,\qquad\;\; D<0\,,
	\end{cases}
	\] 
and $\vr(\Bb)$ can be evaluated in terms of $\FR$ via the Kronecker limit 
formula~\eqref{eq:kronlim3}, this leads to nontrivial identities 
for~$\FR(x)$. We will discuss different realizations of this idea in this section, 
looking first at the case where each genus consists of only one narrow ideal class.

{\def\arraystretch{1.2}
\begin{table}
\begin{tabular}{ r l l }
	\toprule
	$n$ & $n^2-1$ & $4\,\bigl(\JR(n+\sqrt{n^2-1})-n\tfrac{\pi^2}{24}\bigr)$\\
	\midrule
	$2$ & $3$ &             $2S_{1,12}$\\
	$3$ & $2^3$ &           $6S_{1,8}$\\
	$4$ & $3\cdot 5$ &      $2S_{1,15}+4S_{5,12}$\\
	$5$ & $2^3\cdot 3$ &    $3S_{1,24}+2S_{8,12}$\\
	$6$ & $5\cdot 7$ &      $2S_{1,140}+4S_{5,28}$\\
	$7$ & $2^4\cdot 3$ &    $7S_{1,12}+2S_{8,24}$\\
	$8$ & $3^2\cdot 7$ &    $2S_{1,28}+4S_{12,21}$\\
	$9$ & $2^4\cdot 5$ &    $22S_{1,5}+4S_{8,40}$\\
	$11$ & $2^3\cdot 3\cdot 5$ &   $3S_{1,120}+2S_{8,60}+2S_{12,40}+2S_{5,24}$\\
	$12$ & $11\cdot 13$ &          $2S_{1,572}+4S_{13,44}$\\
	$13$ & $2^3\cdot 3\cdot 7$ &   $3S_{1,168}+2S_{8,21}+S_{12,56}+S_{24,28}$\\
	$14$ & $3\cdot 5\cdot 13$ &    $2S_{1,780}+4S_{5,156}+4S_{13,60}$\\
	$16$ & $3\cdot 5\cdot 17$ &    $2S_{1,1020}+4S_{12,85}+4S_{5,204}$\\
	$19$ & $2^3\cdot 3^2\cdot 5$ & $6S_{1,40}+12S_{5,8}+2S_{12,120}+2S_{24,60}$\\
	$21$ & $2^3\cdot 5\cdot 11$ &  $3S_{1,440}+2S_{8,220}+2S_{5,88}+2S_{40,44}$\\
	$22$ & $3\cdot 7\cdot 23$ &    $2S_{1,1932}+2S_{28,69}+2S_{21,92}$\\
	$23$ & $2^4\cdot 3\cdot 11$ &  $4S_{1,33}+2S_{8,264}+S_{12,44}+S_{24,88}$\\
	$27$ & $2^3\cdot 7\cdot 13$ &         $3S_{1,728}+2S_{8,364}+2S_{28,104}+2S_{13,56}$\\
	$29$ & $2^3\cdot 3\cdot 5\cdot 7$ &   $3S_{1,840}+S_{12,280}+2S_{5,168}+S_{24,140}+S_{28,120}+2S_{21,40}+S_{56,60}$\\
	$34$ & $3\cdot 5\cdot 7\cdot 11$ &    $2S_{1,4620}+4S_{5,924}+2S_{28,165}+2S_{60,77}+2S_{21,220}$\\
	$36$ & $5\cdot 7\cdot 37$ &           $2S_{1,5180}+4S_{5,1036}+4S_{37,140}$\\
	$41$ & $2^4\cdot 3\cdot 5\cdot 7$ &   $4S_{1,105}+2S_{8,840}+S_{12,140}+2S_{5,21}+S_{24,280}+S_{28,60}+2S_{40,168}+S_{56,120}$\\
	$43$ & $2^3\cdot 3\cdot 7\cdot 11$ &  $3S_{1,1848}+2S_{8,924}+S_{12,616}+S_{24,77}+S_{28,264}+S_{44,168}+S_{33,56}+S_{21,88}$\\
	$56$ & $3\cdot 5\cdot 11\cdot 19$ &   $2S_{1,12540}+4S_{12,1045}+4S_{5,2508}+2S_{44,285}+2S_{76,165}$\\
	$61$ & $2^3\cdot 3\cdot 5\cdot 31$ &  $3S_{1,3720}+S_{12,1240}+2S_{5,744}+S_{24,620}+2S_{40,93}+S_{60,1032}+S_{120,124}$\\
	$69$ & $2^3\cdot 5\cdot 7\cdot 17$ &  $3S_{1,4760}+2S_{8,2380}+2S_{5,952}+2S_{28,680}+2S_{40,476}+2S_{56,85}+2S_{136,140}$\\
	$77$ & $2^3\cdot 3\cdot 13\cdot 19$ & $3S_{1,5928}+2S_{8,741}+S_{12,1976}+S_{24,988}+2S_{13,456}+S_{76,312}+S_{152,156}$\\
	$83$ & $2^3\cdot 3\cdot 7\cdot 41$ &  $3S_{1,6888}+2S_{8,861}+3S_{12,2296}+S_{24,1148}+S_{28,984}+S_{56,492}+4S_{21,328}$\\
	$131$ & $2^3\cdot 3\cdot 5\cdot 11\cdot 13$ & 
	\begin{tabular}{@{}l@{}}
		$3S_{1,17160}+S_{12,5720}+2S_{5,3432}+S_{24,2860}+2S_{40,429}+S_{44,1560}+2S_{13,1320}$\\
		\quad\quad $+S_{60,1144}+S_{88,780}+2S_{104,165}+S_{120,572}+S_{156,440}+S_{220,312}$
	\end{tabular}\\
	\bottomrule
\end{tabular}
\caption{Values of $\JR(n+\sqrt{n^2-1})$}
\label{tab:jnminus}
\end{table}}

{\def\arraystretch{1.2}
\begin{table}[htbp]
\begin{tabular}{r l l}
	\toprule
	$n$ & $n^2+1$ &$2\,\bigl(\JR(n+\sqrt{n^2+1})-n\tfrac{\pi^2}{24}\bigr)$\\
	\midrule
	$1$ & $2$ &          $S_{1,8}$\\
	$2$ & $5$ &          $4S_{1,5}$\\
	$3$ & $2\cdot 5$ &   $S_{1,40}+2S_{5,8}$ \\
	$4$ & $17$ &         $2S_{1,17}$ \\
	$5$ & $2\cdot 13$ &  $S_{1,104}+2S_{8,13}$\\
	$7$ & $2\cdot 5^2$ & $3S_{1,8}+4S_{5,40}$\\
	$8$ & $5\cdot 13$ &  $2S_{1,65}+4S_{5,13}$\\
	$11$ & $2\cdot 61$ & $S_{1,488}+2S_{8,61}$\\
	$13$ & $2\cdot 5\cdot 17$ &  $ S_{1,680}+ 2S_{8,85} + 2S_{5,136}$\\
	$17$ & $2\cdot 5\cdot 29$ &  $ S_{1,1160}+ 2S_{5,232} + 2S_{29,40}$\\
	$19$ & $2\cdot 181$&         $ S_{1,1448} + 2S_{8,181}$\\
	$23$ & $2\cdot 5\cdot 53$ &  $ S_{1,2120} + 2S_{5,424} + 2S_{40,53}$\\
	$31$ & $2\cdot 13\cdot 37$ & $ S_{1,3848} + 2S_{13,296} + 2S_{37,104}$\\
	$37$ & $2\cdot 5\cdot 137$ & $ S_{1,5480} + 2S_{8,685} + 4S_{5,1096}$\\
	$47$ & $2\cdot 5\cdot 13\cdot 17$ & $ S_{1,8840} + 4S_{5,1768} + 2S_{40,221} + 2S_{13,680} + 2S_{85,104}$\\
	$73$ & $2\cdot 5\cdot 13\cdot 41$ & $ S_{1,21320} + 2S_{5,4264} + 2S_{40,533} + 2S_{13,1640} + 2S_{104,205}$\\
	\bottomrule
\end{tabular}
\caption{Values of $\JR(n+\sqrt{n^2+1})$}
\label{tab:jnplus}
\end{table}}

\subsection{Explicit formulas when there is one class per genus}
In Table~\ref{tab:jnminus} (resp. Table~\ref{tab:jnplus}) we collect all numbers $n$ 
less than $100\,000$ (and conjecturally all~$n$) such that each genus of binary 
quadratic forms of discriminant $4(n^2-1)$ (resp.~$4(n^2+1)$) contains exactly one 
narrow equivalence class. In each case we give an identity for 
$\JR(n+\sqrt{n^2\pm1})$, where
	\[\JR(x) \= 
	J(x)-\frac{\log^2(2)}{2}+\frac{\pi^2}{24}\Big(x-\frac{1}{x}\Big)\]
is related to $\FR(x)$ by
	\[\JR(x) = \FR(2x)-2\FR(x)+\FR(x/2)\,,\]
parallel to the relation~\eqref{eq:jfrel} between $J(x)$ and $F(x)$.
The numbers~$S_{p,q}$ are defined as
	\begin{align*}
		S_{p,q}   &\= \log(\eps_p)\log(\eps_q)\,,
	\end{align*}
where $\eps_d$ for $d$ a fundamental discriminant is the fundamental unit of 
$\QQ(\sqrt{d})$ if $d>1$, and $\eps_d$ is defined to be $2$ if $d=1$.

\begin{theorem} \label{thm:jgenus}
	The values of $\JR(n+ \sqrt{n^2\pm1})$ for the $45$ known one-class-per-genus 
    cases are as given in Table~\ref{tab:jnminus} and Table~\ref{tab:jnplus}. 
\end{theorem}
We observe that, with the exception of $n=7$, 9, 23, and~$41$, the cases listed 
in Table~\ref{tab:jnminus} also follow from Theorem~6 (for $n$ even) or from 
Theorems~8 and~9 (for~$n$ odd) of~\cite{MW}. The identities listed in 
Table~\ref{tab:jnplus} appear to be new.

As already indicated, we will deduce Theorem~\ref{thm:jgenus} from the 
following general expression for $\JR(n+\sqrt{n^2\pm1})$ as a combination 
of Kronecker limits $\vr(\Bb)$.
\begin{theorem} \label{thm:jasrho}
    For all $n\ge1$ the value $\JR(n+\sqrt{n^2\pm1})$
    is a rational linear combination of $\zeta(2)$, $\log(2)\log(n+\sqrt{n^2\pm1})$,
    and the values~$\vr(\Bb)$ for at most four narrow classes~$\Bb$ of quadratic 
    forms of discriminant $4^{a}(n^2\pm1)$, with $a=0, \pm1$.
\end{theorem}
\begin{proof}
We first consider the case of $\JR(u)$ for $u=n+\sqrt{n^2-1}$. Define~$\Bb_1$ 
and~$\Bb_2$ to be the narrow classes with cycles $\cycle{2n}$ and $\cycle{n+1,2}$, 
respectively, and if~$n$ is odd  also define~$\Bb_3$ and~$\Bb_4$ to be the narrow 
classes corresponding to the cycles $\cycle{\frac{n+3}{2},2,2,2}$ 
and $\cycle{\frac{n+1}{2},4}$ respectively. Then we claim that 
	\begin{align}
	\label{eq:ju2}
	2\JR(u) &\= \vr(\Bb_1)-\vr(\Bb_2)+\log(2)\log(u)\,, &&n\text{ even}\,,\\ 
	\label{eq:ju4}
	4\JR(u) &\= \vr(\Bb_1)+\vr(\Bb_2)-\vr(\Bb_3)-\vr(\Bb_4)
	+3\log(2)\log(u)\,,&&n\text{ odd},\, n\ne 7\,.
	\end{align}
We have
	\[\mathrm{Red}(\Bb_1) \= \{u\}\,,
	\quad
	\mathrm{Red}(\Bb_2) \= \Big\{\frac{u+1}{2},\,\frac{2u}{u+1}\Big\}\,,\]
and for $n$ odd also
	\[\mathrm{Red}(\Bb_3) \= 
	\Big\{\frac{4u}{3u+1},\frac{3u+1}{2u+2},\frac{2u+2}{u+3},\frac{u+3}{4}\Big\}
	\,,\qquad \mathrm{Red}(\Bb_4) \= \Big\{\frac{4u}{u+1},\frac{u+1}{4}\Big\}\,.\]
Since $u'=1/u$, we can use the Kronecker limit formula~\eqref{eq:kronlim3} 
to rewrite equation~\eqref{eq:ju2} as
	\[2\JR(u) \= 
	\Pp\bigl(u,\tfrac{1}{u}\bigr)-\Pp\bigl(\tfrac{u+1}{2},\tfrac{u+1}{2u}\bigr)
	-\Pp\bigl(\tfrac{2u}{u+1},\tfrac{2}{u+1}\bigr)+\zeta(2)+\log(2)\log(u)\,,\]
where we recall that $\Pp(x,y)=\FR(x)-\FR(y)+L(y/x)$. This identity, in fact, holds 
for all real $u>1$, as can be derived easily using the functional 
equation~\eqref{eq:heckefe2} and the following simple relation for 
the Rogers dilogarithm
	\[2L\bigl(\tfrac{u}{u+1}\bigr)\+2L\bigl(\tfrac{1}{u}\bigr)
    \m L\bigl(\tfrac{1}{u^2}\bigr) \= 2L(1)\,.\]
Similarly,~\eqref{eq:ju4} is equivalent to
\begin{align*}
	4\JR(u) \= &\Pp\bigl(u,\tfrac{1}{u}\bigr)
	\+2\Pp\bigl(\tfrac{u+1}{2},\tfrac{u+1}{2u}\bigr)
	\m2\Pp\bigl(\tfrac{u+1}{4},\tfrac{u+1}{4u}\bigr)\\
	&\m2\Pp\bigl(\tfrac{4u}{3u+1},\tfrac{4}{u+3}\bigr)
	\m2\Pp\bigl(\tfrac{3u+1}{2u+2},\tfrac{u+3}{2u+2}\bigr)
	\+3\zeta(2)\+3\log(2)\log(u)\,.
\end{align*}
Again, this identity holds for all $u>1$, and can be derived using the 
functional equations~\eqref{eq:heckefe2} and~\eqref{eq:heckefe4} for $\FR(x)$.
Finally, for $n=7$ we have
	\[4\JR(7+\sqrt{48}) \= 
	\vr(\Bb_1)-\vr(\Bb_2')+2\zeta(2)+\tfrac{7}{2}\log(2)\log(7+\sqrt{48})\,,\]
where $\Bb_2'$ corresponds to the cycle $\cycle{6,2,2}$.
This can again be derived using~\eqref{eq:heckefe2} and~\eqref{eq:heckefe4}.

The case of $\JR(v)$ for $v=n+\sqrt{n^2+1}$ is similar. 
Here we define $\Aa_1$ and~$\Aa_2$ to be the wide ideal classes with
cycles $\cyclesq{2n}$ and $\cyclesq{n-1,1,1}$ respectively. We claim 
that 
	\bes
	4\JR(v) \= \vr(\Aa_1)-\vr(\Aa_2)+2\log(2)\log(v)+n\zeta(2)\,,
	\qquad\qquad n>2\,.
	\ees
Once again, using
	\[\Red_w(\Aa_1) \,=\, \{v\}\,,
	\qquad
	\Red_w(\Aa_2) \,=\, 
	\big\{\tfrac{2v}{v+1},\,\tfrac{v+1}{v-1},\,\tfrac{v-1}{2}\big\}\,,\]
and the Kronecker limit formula~\eqref{eq:kronlimwide}
one can rewrite the identity as
	\[2\JR(v) \= \widetilde\Pp\bigl(v,\tfrac{1}{v}\bigr)
	-\widetilde\Pp\bigl(\tfrac{2v}{v+1},\tfrac{2}{v-1}\bigr)
	-\widetilde\Pp\bigl(\tfrac{v+1}{v-1},\tfrac{v-1}{v+1}\bigr)
	-\widetilde\Pp\bigl(\tfrac{v-1}{2},\tfrac{v+1}{2v}\bigr)-\frac{1}{2}\zeta(2)
	+\log(2)\log(v)\,,\]
(with $\widetilde \Pp$ as in~\eqref{eq:p3term}) which follows easily 
from~\eqref{eq:heckefe2} and the $3$-term relation~\eqref{eq:3term2}. 
Finally, the cases $n=1,2$ do not fit the above scheme, but can be derived directly 
using functional equations that $\JR(n+\sqrt{n^2+1})$ is a linear combination 
of $\zeta(2)$ and $\log(2)\log(n+\sqrt{n^2+1})$. 
For instance, the case $n=1$ follows directly by substituting $x=1+\sqrt{2}$ 
into~\eqref{eq:heckefe2} and applying the $3$-term relation~\eqref{eq:3term2} 
with $x=\sqrt{2}$. The derivation for $n=2$ is slightly more complicated, but 
again only involves the functional equations~\eqref{eq:3term2} 
and~\eqref{eq:heckefe2}.
\end{proof}

\begin{proof}[Proof of Theorem~\ref{thm:jgenus}]
	Since in all the cases listed in the tables there is one class per genus
	of quadratic forms, we can rewrite the corresponding linear combination
	of $\vr(\Bb_i)$ as a rational linear combination 
	of $D^{1/2}L_{\mathcal{O}_D}(1,\chi)$, where $\chi$ runs over genus 
	characters of the corresponding class groups.
	In each case the identity then follows from the factorization
	of $L_{\mathcal{O}}(s,\chi)$ into a product of two Dirichlet 
	$L$-functions (for general discriminants, see, for example,~\cite{KM}).
	
	There is a small subtlety in some cases when we need to consider a 
    combination $\vr(\Bb)-\vr(\Bb')$ with~$\Bb$ of discriminant~$D$ and~$\Bb'$ 
    of discriminant~$4D$ and a priori the resulting expression for 
	$\JR(n+\sqrt{n^2\pm1})$ may involve constant terms of $\zeta_{\Oo_D}(s)$ and 
	$\zeta_{\Oo_{4D}}(s)$ at $s=1$. 
	However, a simple but somewhat tedious calculation using the explicit 
	expression for $\zeta_{\mathcal{O}}(s)$ (see, say~\cite{K}) and the 
	relation between the class groups of discriminants~$D$ and~$4D$ 
	(see,~e.g.,~\cite[Cor.~5.9.9]{HK}) shows that 
	in all these cases the nontrivial contributions coming
    from $\zeta_{\mathcal{O}}(s)$ cancel out.
\end{proof}

Note that to say that~$\eps$ is a number of the form $n+\sqrt{n^2\pm1}$ is equivalent 
to saying that $\eps>1$ is a quadratic unit with even trace (or a unit in an order of 
even discriminant). As far as we can ascertain, there are no similar formulas
for~$\JR(\eps)$ when~$\eps$ is a unit of odd trace.

\subsection{General case}
The proof of Theorem~\ref{thm:jgenus} shows that $\JR(n+\sqrt{n^2\pm1})$ 
can always be expressed as an algebraic linear combination of $\zeta(2)$, 
$\log(2)\log(n+\sqrt{n^2\pm1})$, and terms of 
the form $D^{1/2}L_{\Oo_D}(1,\chi)$, where $D=4^{a}(n^2\pm1)$, 
$a\in\{0,\pm1\}$, and $\chi$ is a narrow class group character.

Recall that the Stark conjecture predicts that the special value at $s=1$ 
of the Artin \hbox{$L$-function} of a Galois representation $\rho$ for a Galois 
extension $E/F$ is a simple multiple (a power of $\pi$ times an algebraic 
number) of the so-called Stark regulator (the determinant of a certain matrix 
of logarithms of units) in such a way that the
factorization $\zeta_E(s)=\prod_{\rho}L_{E/F}(s,\rho)^{\dim(\rho)}$ 
of the Dedekind zeta function of~$E$ matches the factorization 
of the regulator of~$E$ obtained by decomposing the group of units of~$E$ 
(after extending scalars to $\ol\QQ$) into irreducible 
${\rm Gal}(E/F)$-representations. 
For the general formulation of the Stark conjecture, see~\cite[p.~25--28]{T}.

In the case of an abelian extension $H/K$ (in our cases, a ring class field) 
of a real quadratic field~$K$, all irreducible representations of 
${\rm Gal}(H/K)$ are one-dimensional, all irreducible representations
of ${\rm Gal}(H/\QQ)$ are one- or two-dimensional, and for any character $\chi$ 
of ${\rm Gal}(H/K)$ we have $L_{H/K}(s,\chi)=L_{H/\QQ}(s,\rho)$, where~$\rho$
is the two-dimensional representation of ${\rm Gal}(H/\QQ)$ induced 
from~$\chi$. The Stark regulator for $L_{H/\QQ}(1,\rho)$ is then a $2\times 2$,
$1\times 1$, or $0\times 0$ determinant of logarithms of units in $H$,
depending on whether $\mathrm{tr}(\rho(\s))$ is $2$, $0$, or $-2$, where 
$\s$ denotes the complex conjugation.

If $H$ is the ring class field of a real quadratic field $K$ corresponding to an 
order $\Oo_D\subset K$ and~$\chi$ a character on the corresponding narrow ideal 
class group, then $L_K(s,\chi)$ is the $L$-function of the corresponding character 
on ${\rm Gal}(H/K)$, and since $H$ is totally real, the Stark conjecture predicts 
that $L_K(1,\chi)$ should be an algebraic multiple of a $2\times 2$ determinant 
of units in~$H$ (compare with~\cite[p.~61]{S}). We therefore obtain the following result.

\begin{proposition}
Let $\eps>0$ be a quadratic unit with even trace. 
Assume the abelian Stark conjecture for the 
real quadratic field $\QQ(\eps)$. Then $\JR(\eps)$ is a rational linear 
combination of $\zeta(2)$, $\log(2)\log(\eps)$, and of $2\times 2$ determinants 
of algebraic units in the narrow ring class field of the quadratic order $\ZZ[2\eps]$.
\end{proposition}

Here is an explicit numerical example. 
The field $K=\QQ(\sqrt{257})$ has class number~$3$ and its 
Hilbert class field $H$ is $K(\alpha)$, where $\alpha$ satisfies 
$\alpha^3-2\alpha^2-3\alpha+1=0$. If we let $\alpha_1<\alpha_2<\alpha_3$
be the roots of $x^3-2x^2-3x+1$, then to very high precision we find
	\[\JR(16+\sqrt{257}) \stackrel{?}{\=} 4\zeta(2)
	+\log(2)\log(16+\sqrt{257})
	+ 2\begin{vmatrix}
		\log(-\alpha_1) & \log(3-\alpha_1) \\
		\log(\alpha_3) & \log(3-\alpha_3) \\
	   \end{vmatrix}
	\,.\]
As another example, this time with a non-fundamental discriminant, we find
	\[2\,\JR(10+3\sqrt{11}) \stackrel{?}{\=} 5\zeta(2)
	+ \log(2)\log(10+3\sqrt{11})
	+ \begin{vmatrix}
		\log(\beta_1) & \log(\gamma_1) \\
		\log(\beta_2) & \log(\gamma_2) \\
	\end{vmatrix}
	\,,\]
where $\beta_1>\beta_2$ are the two largest real roots of $x^4-11x^3+24x^2-11x+1$ 
and $\gamma_j=\frac{\beta_j^2-9\beta_j+4}{1-\beta_j}$.

\section{Cohomological aspects} \label{sec:cohomology}
Note that the $3$-term relation
	\be \label{eq:f3term1}
	\FR(x)\m\FR(x+1)\m\FR\Big(\frac{x}{x+1}\Big) \=
	L(1/2)-L\Big(\frac{x}{x+1}\Big)\,, 
	\ee
when viewed modulo the more elementary right-hand side is exactly
the period relation satisfied by the component $\phi_S$ 
of a 1-cocycle $\{\phi_{\gamma}\}_{\gamma\in\PSL_2(\ZZ)}$ such that $\phi_T=0$. 
(Cf.~\cite{Z3} and also~\cite{LZ} for more discussion of this relationship.)
In this section we will discuss various ways---all of them still somewhat 
provisional---in which one can relate~$\FR$ to 1-cocycles for $\PSL_2(\ZZ)$.

\subsection{An interplay between the 5-term relation and the 3-term relation} \label{sec:5term}
If we extend $\FR$ to an even function on $\RR\sm\{0\}$, then the 
weaker form of relation~\eqref{eq:f3term1}
	\be \label{eq:f3term}
	\FR(x)-\FR(x-1)+\FR\Big(\frac{x-1}{x}\Big) \= L(x)-L(1/2)\Mod{\zeta(2)}
	\ee
now holds for all $x\in\RR\sm\{0,1\}$. 
The function $\Pp(x,y)$ defined by~\eqref{eq:Pdef}
then becomes an even piecewise continuous function on all of $\RR^2$ and
satisfies the functional equation
	\be \label{eq:p3termB}
	\Pp(x,y)-\Pp(x-1,y-1)
	+\Pp\Big(\frac{x-1}{x},\frac{y-1}{y}\Big) \= 0 \Mod{\zeta(2)}
	\ee
for all $x,y\in\RR\sm\{0,1\}$.
Conversely, if $\FR$ is any function and we define $L$ and $\Pp$ (modulo constants) 
by equations~\eqref{eq:f3term} and~\eqref{eq:Pdef}, respectively, then the left hand 
side of~\eqref{eq:p3termB} becomes a sum of $15$~$\FR$'s that can be grouped into $5$ 
$L$'s and then becomes the famous 5-term relation  (equivalent to~\eqref{eq:5term1}) 
for~$L$:
	\be \label{eq:5term2}
	L\Big(\frac{y}{x}\Big)
	-L\Big(\frac{y-1}{x-1}\Big)
	+L\Big(\frac{1-1/y}{1-1/x}\Big)
	-L(y)+L(x) \= 0 \Mod{\zeta(2)}\,.
	\ee
	
The anti-invariance of~$\FR$ and $L$ under inversion implies that~$\Pp$ is also anti-invariant 
and this together with~\eqref{eq:p3termB} implies that~$\Pp$ defines a cocycle for the 
group $\PSL_2(\ZZ)$ with values in the space of functions from $\RR^2$ to $\RR/\zeta(2)\ZZ$, 
by mapping~$T$ to~$0$ and~$S$ to~$\Pp$.

The above discussion thus gives an interesting connection, which we have not yet 
understood completely, between the 5-term relation and 1-cocycles for the group $\PSL_2(\ZZ)$.

\subsection{A cocycle with values in $\Lambda^2(\QQ^{\mathrm{ab}\, \times})$} \label{sec:k2cocycle}
A different construction of a 1-cocycle can be obtained from the evaluation of 
$F$ at positive rationals given by Theorem~\ref{thm:Fspecial}.
The formula that we established in the proof has the form
	\[F(\alpha)-F(1) \= \Li_2(\xi_{\alpha})\,,\] 
for a certain $\xi_{a}\in\ZZ[\QQ(e^{2\pi i \alpha}, e^{2\pi i/\alpha})]$.
It is therefore interesting to know when a linear combination of $\xi_{\alpha}$'s
lands in the Bloch group of~$\ol\QQ$ (see~\cite{Z4}).
For this we have to calculate $\delta(\xi_{\alpha})$, where 
$\delta\colon\ZZ[\ol\QQ] \to \Lambda^2(\ol\QQ^{\times})\otimes_{\ZZ}\QQ$ is the usual 
differential given 
by $\delta([x])=x \wedge (1-x)$. We find that if $\alpha=p/q$, $(p,q)=1$, then 
	\[\delta(\xi_{\alpha}) \= \beta(\alpha)-\beta(1/\alpha)-p\wedge q\,,\] 
where $\beta\colon\QQ\to\Lambda^2(\QQ^{\text{ab}\, \times})\otimes_{\ZZ}\QQ$ is 
defined by 
	\be \label{eq:beta} \beta(p/q) \= 
    \sum_{\lambda^q=1}(1-\lambda^p)\wedge (1-\lambda) \qquad ((p,q)=1)\,.\ee
	
We further note that the element $\xi_{\alpha}$ is Galois invariant, 
so whenever $\delta(\xi_{\alpha})$ vanishes, by Galois descent we get an element 
in the Bloch group of $\QQ$, which is a torsion group. Since $K_2(\QQ)$ is torsion 
(see Theorem~11.6 in~\cite{Mi}) the group $\Lambda^2(\QQ^{\times})\otimes_{\ZZ}\QQ$ 
is equal to $\delta(\QQ[\QQ])$. Therefore, whenever a linear 
combination $\sum_ic_i(\beta(\alpha_i)-\beta(1/\alpha_i))$ vanishes, 
there are rational numbers~$x_j$ and~$d_j$ such 
that $\sum_i c_i \xi_{\alpha_i} + \sum d_j [x_j]$ 
is in $\ker(\delta)$, and then the sum $\sum_ic_iF(\alpha_i)+\sum_jd_j\Li_2(x_j)$ 
vanishes modulo products of logarithms of algebraic numbers. A rather simple example 
of this was given in~\eqref{eq:f1n}. As a slightly less trivial example, one can 
show that $\beta(\frac{n}{n^2+1}) = \beta(\frac{n^2+1}{n})$, which implies that 
$F(\frac{n}{n^2+1})-F(1)-\frac{1}{2}\Li_{2}(\frac{n^2}{n^2+1})$
is a bilinear combination of logarithms of algebraic numbers, 
generalizing~\eqref{eq:f25}.

Two remarks about the function $\beta$ appearing above are that it takes 
values in $\Lambda^2$ of the group of cyclotomic units and that it satisfies the two 
identities $\beta(-x)=\beta(x)$ and $\beta(x+1)=\beta(x)$. Because of this, sending 
$T\mapsto 0$, $S\mapsto \phi$, where $\phi(p/q)=\delta(\xi_{p/q})+p\wedge q$ defines 
a $\PSL_2(\ZZ)$-cocycle. This cocycle can also be lifted to a cocycle given by 
$T\mapsto 0$, $S\mapsto \xi$ with values in functions from $\PP^1(\QQ)$ 
to $\QQ[\QQ^{\text{ab}}]/(\QQ[\QQ]+\mathcal{C}(\QQ^{\text{ab}}))$, where 
$\mathcal{C}(F)$ denotes the subspace of $\QQ[F]$ spanned by all specializations of 
the 5-term relation.

A further remark is that~$\beta$ itself can be used to construct 
infinitely many 1-cocycles, now with values in 
$\Lambda^2(\QQ(\zeta_p)^{\times})^{\PP^1(\mathbb{F}_p)}$ 
for any prime $p$: the mapping
$\beta_p\colon\ZZ/p\ZZ\to\Lambda^2(\QQ(\zeta_p)^\times)\otimes_{\ZZ}\QQ$ 
given by $\beta_p(n)=\beta(n/p)$ satisfies the functional 
equations~$\beta_p(n)=\beta_p(-n)=-\beta_p(1/n)=\beta_p(n+1)+\beta_p(n/(n+1))$, 
and thus sending $T\mapsto 0$, $S\mapsto \beta_p$ defines 
a $\PSL(2,\mathbb{F}_p)$-cocycle.

A final remark is that there is a formal similarity between the 
formula~\eqref{eq:beta} and the classical Dedekind sums arising in 
the modular transformation behavior of the Dedekind eta function.
This similarity can be made precise using the notion of generalized Dedekind symbols
due to Fukuhara~\cite{F}: the mapping $(p,q)\mapsto \beta(p/q)$ is an even 
generalized Dedekind symbol with values in $\Lambda^2(\QQ^{\mathrm{ab}\, \times})$, 
and the mapping $(p,q)\mapsto \delta(\xi_{p/q})+p\wedge q$ is its reciprocity 
function. Note that from this point of view the functional equations discussed in 
Section~\ref{sec:funceq} (or, more precisely, the corresponding functional equations 
for~$\beta$) are analogous to the functional equations discovered by Knopp~\cite{Kn} 
for the classical Dedekind sums.

\subsection{The Herglotz function and the weight 2 Eisenstein series} \label{sec:period}
In this final subsection, we will give an explanation of the cocycle nature of~$F$ by 
writing it as an Eichler-type integral. We slightly modify the definition of $F$ to
    \[ F^\star(z) \= \sum_{n\ge1}\frac{\psi(nz)-\log(nz)+(2nz)^{-1}}n \= F(z) 
    \+\frac{\pi^2}{12\,z}\,.\]
Using Binet's integral formula~\cite[p.~250]{WW} for the digamma function,
    \[\psi(x) \= \log x \,-\, \frac{1}{2x} \,-\, \int_{0}^{\infty}\frac{1}{e^{2\pi 
    t}-1}\,\frac{2t\,dt}{t^2+x^2}\,,\]
we obtain the following integral representation, valid for $\mathrm{Re}(z)>0$:
	\[ F^{\star}(z) \= -\int_{0}^{\infty} 
	\Biggl(\,\sum_{n=1}^{\infty}\frac1{n(e^{2\pi nt}-1)}\Biggr)\frac{2t\,dt}{t^2+z^2}
	\= -\int_{0}^{i\infty}H(\t) \Bigl(\frac{1}{\t+z}+\frac{1}{\t-z}\Bigr)\,d\t\,, \]
where $H(\t)=\sum_{n=1}^{\infty}\s_{-1}(n)\,q^n$. (Here, as usual, we use $q$ to 
denote $e^{2\pi i \t}$ and $\s_\nu(n)$ for the sum of the $\nu$-th powers of the 
positive divisors of~$n$.) Note that $H(\t) = \log(q^{1/24}/\eta(\t))$,
where~$\eta$ is the Dedekind eta function, and that its derivative is $2\pi i\, 
G_2^0(\t)$, where $G_2^0(\t)$ is the weight~2 Eisenstein series 
$G_2(\t)=-\frac1{24}+\sum_{n=1}^{\infty}\s_1(n)\,q^n$ 
without its constant term. Integrating by parts, we therefore get
    \[ F^\star(z) \= 
    2\pi i \int_{0}^{i\infty}G_{2}^{0}(\t)\log\Big(1-\frac{\t^2}{z^2}\Big)\,d\t\,. \]
This formula is valid as it stands for $\re(z)>0$ and gives yet another proof
of the analytic continuation of~$F$ to~$\CC'$ by deforming the path of integration to 
remain to the left of~$z$ if~$z$ is in the second quadrant and to the right of~$-z$ 
if~$z$ is in the third quadrant.

The above integral for $F^\star$ can be rewritten as
$F^\star(z)=F^{+}(z)+F^{-}(z)$, where
    \[ F^{\pm}(z) \= 2\pi i \int_{0}^{i\infty}G_{2}^{0}(\t)
    \,\Bigl(\log\Bigl(1\mp\frac{\t}{z}\Bigr)\pm\frac{\tau}{z}\Bigr)\,d\t\,. \]
Note that both $F^{+}$ and $F^{-}$ continue analytically to $\CC'$. If we denote 
by $\Hf^{+}$ and $\Hf^{-}$ the upper and lower half-planes respectively, then by 
splitting the integral from $0$ to $i\infty$ at $z$ or $-z$ we get
    \[F^{\pm}(z) \;\equiv\; H^{\pm}(z)\m H^{\pm}(-1/z) \+ \log(z)H(\pm z) \m 
    U(z^{\pm1}) \qquad \bigl(z\in\Hf^{\pm}\bigr)\,,\]
where ``$\equiv$'' means ``modulo elementary functions'', 
$U(z)=2\pi i \int_{z}^{\infty}G_2^{0}(\t)\log(\t)d\t$, and
    \[H^{\pm}(z) \= \pm 2\pi i \int_{z}^{\pm i\infty}G_{2}^{0}(\pm\t)\log(\t-z)\,d\t  
    \qquad (z\in\Hf^{\pm})\,.\]
Since $H^{\pm}(z)$ is obviously a 1-periodic function in $\Hf^{\pm}$ and $U(z)$ is 
essentially the primitive of $G_2(z)\log(z)$, this computation implies that 
$F^{\pm}(z)-F^{\pm}(z+1)-F^{\pm}(z/(z+1))=0$ modulo elementary functions, and that
$T\mapsto 0$, $S\mapsto F^{\pm}$ defines a 1-cocycle with values in the space of 
suitably nice analytic functions modulo elementary functions.

The above calculations go through much the same way for the higher Herglotz functions
	\[\mathscr{F}_k(z) \= \sum_{n\ge1}\frac{\psi(nz)}{n^{k-1}}\qquad (k>2)\,,\]
with~$H$ replaced by $\sum_{n\ge1}\s_{1-k}(n)q^n$, which for even $k$ is essentially 
the Eichler integral of the Eisenstein series of weight~$k$ on the full modular 
group. The functions $\mathscr{F}_k$ were defined in~\cite{VZ} in connection with a 
higher Kronecker ``limit'' formula for $\zeta_{K}(\Bb,s)$ at $s=k/2$ (for~$k$ even) 
rather than the limiting value at~$s=1$. It is likely that most of the properties we 
have given have analogues for higher Herglotz functions. This could be a topic for 
future research.


\begin{thebibliography}{99}
\bibitem{CZ} Y.-J.~Choie, D.~Zagier, \emph{Rational period functions for 
$\operatorname{PSL}(2,\ZZ)$}, in A Tribute to Emil Grosswald: Number Theory and 
Related Analysis, Contemp. Math.~\textbf{143}, AMS, Providence, RI, 1993, pp.~89--108.

\bibitem{C} H.~Cohen, \emph{Number Theory. Analytic and Modern Tools}. Graduate Texts 
in Mathematics, vol.~II. Springer, New York, 2007.

\bibitem{F} S.~Fukuhara, \emph{Hecke operators on weighted Dedekind symbols}, 
J.~reine angew. Math. \textbf{593}, pp.~1--29 (2006).

\bibitem{HK} F.~Halter-Koch, \emph{Quadratic Irrationals. An Introduction to 
Classical Number Theory}, Pure and Applied Mathematics, CRC Press, Boca Raton, 2013.

\bibitem{H} G.~Herglotz, \emph{\"Uber die Kroneckersche Grenzformel f\"ur reelle, 
quadratische  K\"orper~I},  Ber. Verhandl. S\"achsischen Akad. Wiss. Leipzig 
\textbf{75}, pp.~3--14 (1923).

\bibitem{K} M.~Kaneko, \emph{A generalization of the Chowla-Selberg formula and the 
zeta functions of quadratic orders}, Proc. Japan Acad. Ser. A Math. Sci.~\textbf{66}, 
pp.~201--203 (1990).

\bibitem{KM} M.~Kaneko, Y.~Mizuno, \emph{Genus character L-functions of quadratic 
orders and class numbers}, J. London Math. Soc.~\textbf{102(1)}, pp.~69--98 (2020).

\bibitem{Kn} M.~Knopp, \emph{Hecke operators and an identity for the Dedekind sums}, 
J. Number Theory~\textbf{12}, pp.~2--9 (1980).

%\bibitem{L} J.~Lewis, \emph{Spaces of holomorphic functions equivalent to the even 
%Maass cusp forms}, Invent. math.~\textbf{127}, pp.~271--306 (1997).

\bibitem{LZ} J.~Lewis, D.~Zagier, \emph{Period functions for Maass wave forms. I}, 
Ann. of Math.~\textbf{153}, pp.~191--258 (2001).

\bibitem{Ma} Yu.~Manin, \emph{Periods of parabolic forms and p-adic Hecke series}, 
Math. USSR-Sb.~\textbf{21:3}, pp.~371--393 (1973).

\bibitem{Me} L.~Merel, \emph{Universal Fourier expansions of modular forms}, in: 
On Artin's Conjecture for Odd 2-dimensional Representations, Springer, 1994, 
pp.~59--94.

\bibitem{Mi} J.~Milnor, \emph{Introduction to Algebraic K-theory}, Annals of 
Mathematics Studies vol.~\textbf{72}, Princeton University Press, Princeton, NJ, 1971.

\bibitem{MW} H.~Muzaffar, K.~S.~Williams, \emph{A restricted Epstein zeta function 
and the evaluation of some definite integrals}, Acta Arithm.~\textbf{104(1)}, 
pp.~23--66 (2001).

%\bibitem{PARI2} The~{PARI} Group, \emph{PARI/GP version {\tt 2.13.0}}, 2020. http://pari.math.u-bordeaux.fr/.

\bibitem{S} H.~M.~Stark, \emph{L-functions at $s=1$. II. Artin L-functions with 
rational characters}, Adv.~Math. \textbf{17}, pp.~60--92 (1975).

\bibitem{T} J.~Tate, \emph{Les Conjectures de Stark sur les Fonctions~$L$ d'Artin en 
$s=0$}, Progress in Mathematics, Vol.~47, Birkh\"auser, Boston-Basel-Stuttgart, 1984.

\bibitem{VZ} M.~Vlasenko, D.~Zagier, \emph{Higher Kronecker ``limit'' formulas for 
real quadratic fields}, J.~reine angew. Math.~\textbf{679}, pp.~23--64 (2013).

\bibitem{WW} E.~T.~Whittaker, G.~N.~Watson, \emph{A Course of Modern Analysis}, 
Cambridge Mathematical Library, Cambridge University Press, Cambridge, 1996. Reprint 
of the fourth (1927) edition.

\bibitem{Z1} D.~Zagier, \emph{A Kronecker limit formula for real quadratic fields}, 
Math.~Ann.~\textbf{213}, pp.~153--184 (1975).

\bibitem{Z2} D.~Zagier, \emph{Zetafunktionen und quadratische K\"orper. Eine 
Einf\"uhrung in die h\"ohere Zahlentheorie}, Springer, Berlin-New York, 1981.

\bibitem{Z3} D.~Zagier, \emph{Quelques cons\'{e}quences surprenantes de la 
cohomologie de $\SL(2,\ZZ)$}, Le\c{c}ons de math\'{e}matiques d'aujourd'hui, Cassini, 
Paris, 2000, pp.~99--123.

\bibitem{Z4} D.~Zagier, \emph{The dilogarithm function}, in Frontiers in Number 
Theory, Physics, and Geometry, Vol.~II, pp.~3--65. Springer, Berlin, 2007.

\bibitem{Z5} D.~Zagier, \emph{Curious and exotic identities for Bernoulli numbers}, 
in: T. Arakawa et al., Bernoulli Numbers and Zeta Functions, Springer Monographs in 
Mathematics, 2014, pp.~239--267.
\end{thebibliography}
\end{document}